\documentclass[a4paper,10pt]{amsart}

\usepackage{amscd}
\usepackage{amsthm}
\usepackage[centertags]{amsmath}
\usepackage{amsfonts}
\usepackage{newlfont}
\usepackage{graphicx}
\usepackage[usenames]{color}
\usepackage{amsfonts, amssymb}
\usepackage{mathrsfs}
\usepackage{latexsym}
\usepackage{tikz}\usepackage{verbatim}
\usepackage{tabulary,tabularx}
\usepackage[all]{xy}
\usepackage[latin1]{inputenc}
\usepackage{enumitem}
\usepackage [all]{xy}

\newtheorem{theorem}{Theorem}[section]

\newtheorem{proposition}[theorem]{Proposition}

\newtheorem{lemma}[theorem]{Lemma}
\newtheorem{definition}[theorem]{Definition}

\theoremstyle{definition}

\newtheorem{example}[theorem]{Example}

\theoremstyle{definition}

\theoremstyle{remark}

\newcommand{\LL}{\mathcal L}

\newcommand{\zN}{\mathbb N}

\renewcommand{\H}{H(\mathbb{C})}
\newcommand{\C}{\mathbb C}

\newcommand{\sub}{\subseteq}

\newcommand{\zD}{\mathbb D}

\newcommand{\f}{\frac}

\renewcommand{\l}{\left(}
\renewcommand{\r}{\right)}

\title{Hypercyclic bilinear operators on Banach spaces}
\author{Rodrigo Cardeccia}

\address{DEPARTAMENTO DE MATEM\'ATICA - PAB I,
	FACULTAD DE CS. EXACTAS Y NATURALES, UNIVERSIDAD DE BUENOS AIRES, (1428) BUENOS AIRES, ARGENTINA AND IMAS-CONICET} \email{rcardeccia@dm.uba.ar} 

\thanks{Partially supported by ANPCyT PICT 2015-2224, UBACyT 20020130300052BA, PIP 11220130100329CO and CONICET}

\begin{document}

\begin{abstract}
	We study the dynamics induced by an $m$-linear operator. 
	We answer a question of B\`es and Conejero showing an example of an $m$-linear hypercyclic operator acting on a Banach space. Moreover, we prove the existence of $m$-linear hypercyclic operators on arbitrary infinite dimensional separable Banach spaces. We also prove an existence result about symmetric bihypercyclic bilinear operators, answering a question by Grosse-Erdmann and Kim. 
\end{abstract}
\maketitle
\section{Introduction}
Given a Fr\'echet space $X$, a linear operator $T$ is called hypercyclic provided that there is a vector $x$ such that its induced orbit $Orb_T(x):=\{T^n(x):n\in\zN\}$ is dense in $X$. The 
first example of a hypercyclic operator is the translation operator $\tau_1(f)=f(1+\cdot)$ on  $\H$, the space of complex analytic functions, and was found by Bikhoff \cite{Bir29} in 1929. Later, some other natural examples appeared, like the MacLane operator, $D(f)=f'$ also on $\H$ \cite{Mac52}, the Rolekwicz operator $2B$ on $\ell_p$ \cite{Rol69} ($B$ denotes the backward shift operator),  among others. However it was not until 
the 80's that a systematic treatment on the subject began.   Evidences on the maturity reached in the area are the survey \cite{Gro99} and the books \cite{BayMat09,GroPer11}. 
In the last decades linear dynamics has experienced a lively development and it seems natural to extend the notion to the iteration of non-linear mappings.

The first to study dynamics of homogeneous polynomials on Banach spaces was Bernardes \cite{Ber98}. Maybe surprisingly he showed that no (non linear) homogeneous polynomial is hypercyclic if the space is Banach. The reason is that every homogeneous polynomial supports an invariant ball (afterwards \textit{the limit ball}) at the origin. 

On the other hand, if the space is not normable then it can support hypercyclic homogeneous polynomials. The first who realized this fact was Peris, 
who exhibited an example of a hypercyclic homogeneous polynomial on $\C^\zN$, see \cite{Per99,Per01}.
Later on, some other examples appeared, in some spaces of differentiable functions over the real line \cite{AroMir08}, some K\"othe Echelon spaces (including $H(\zD)$) \cite{MarPer10} and recently in $\H$ \cite{CardMurH(C)}. 
 
Grosse-Erdmann and Kim \cite{GroKim13} generalized the notion of hypercyclicity to bilinear operators, and showed that, in some sense, the  limit ball problem (which is an obstruction for homogeneous polynomials to be hypercyclic) may be avoided. Let us recall their definition. Given a Banach space $X$ and $x,y\in X$ the orbit of a bilinear mapping $M\in\mathcal L(^2X;X)$ with initial conditions $(x,y)$ is $\cup_{n\ge 0} M^n$ where the $n$- states  $M^n(x,y)$ are inductively defined as $M^0(x,y)=\{x,y\}$ and $M^n(x,y)=M^{n-1}(x,y)\cup \{M(z,w):z,w\in M^{n-1}\}$. A bilinear operator is said to be \textit{bihypercyclic} provided that some orbit is dense in $X$. In \cite{GroKim13}, some nice results concerning bihypercyclic operators were obtained. For example, the set of bihypercyclic vectors is always $G_\delta$ but never residual. They also succeeded to construct bihypercyclic bilinear operators (not necessarily symmetric) in arbitrary separable Banach spaces (including the finite dimensional case). However it is unknown whether the operator can be taken to be symmetric and the following question was posed (see \cite[p. 708]{GroKim13}).

	\textbf{Question A.}\textit{
	Let $X$ be a separable Banach space. Does there exist a symmetric bihypercyclic operator in $\LL(^2X)$?}

Nevertheless the definition of the orbit induced by a multilinear operator is not canonic and other interpretations are available. Whereas the $n$-state of the iterate of a linear operator depends only on the immediately preceding step ($x_n=T(x_{n-1})$), it would be desirable that  the $n$-state of the iterate of an $m$-linear operator depends only on the $m$-previous steps. B\`es and Conejero \cite{BesCon14} defined the orbit induced by a multilinear operator $M$ with initial conditions $x_{1-m},\ldots x_0$ as $Orb_M(x_{m-1},\ldots x_0)=\cup_{n} \{x_n\}$, where each $x_n$ is inductively defined as $x_n=M(x_{n-m},\ldots,x_{n-1})$. A multilinear operator is said to be hypercyclic if there are $x_{1-m},\ldots,x_0\in X$ such that $Orb_M(x_{1-m},\ldots x_0)$ is dense in $X$.  Since the orbit in the sense of B\`es and Conejero is contained in the orbit in the sense of Grosse-Erdmann and Kim it follows that a hypercyclic bilinear operator is automatically bihypercyclic. This contention implies also that there is again a sense of limit ball for Banach spaces, when $m$-consecutive vectors are in $\f{1}{\|M\|^{m-1}} B_X$, the orbit tends to zero and therefore the set of hypercyclic vectors is never residual. In \cite{BesCon14} examples of multilinear operators over non normable Fr\'echet spaces where given, including $\H$ and $\C^\zN$. It was also proved that every infinite dimensional and separable Banach space supports a supercyclic multilinear operator (i.e. $\overline{\C Orb_M(x_{1-m,\ldots ,x_0})}=X$). However no example of a hypercyclic multilinear operator on a Banach space or without a residual set of hypercyclic vectors was given and thus the following questions were posed in \cite[Section 5]{BesCon14}.
 
\textbf{Question B.} 
\textit{
Let $X$ be a Fr\'echet space and $M$ a hypercyclic multilinear operator. Is the set of hypercyclic vectors necessarily residual? }
 
\textbf{Question C.} \textit{
Are there hypercyclic multilinear operators acting on  Banach spaces?}

Of course a positive answer for Question B implies a negative answer for  Question C.	
 
The structure of the paper is the following. In Section \ref{SecRes} we propose a notion of transitivity for multilinear operators and analyze examples of multilinear hypercyclic operators over non normable  Fr\'echet spaces with and without a residual set of hypercyclic vectors. In particular we answer Question B by showing a multilinear hypercyclic operator without a residual set of hypercyclic vectors. In Section \ref{SecBan} we answer Question C positively. Moreover we construct bilinear hypercyclic operators in arbitrary separable and infinite dimensional Banach spaces. In Section \ref{SecBih} we answer Question A posed by Grosse-Erdmann and Kim \cite{GroKim13}, proving that there are symmetric bihypercyclic operators in arbitrary separable and infinite dimensional Banach spaces.

\section{Bilinear hypercyclic operators on non normable Fr\'echet spaces}\label{SecRes}
%
The orbit of an $m$-linear operator $M$ with initial condition $(x_{1-m},\dots,x_0)$ was defined in \cite{BesCon14} as the set
$$
Orb_M(x_{1-m},\dots,x_0)=\cup_{n\geq 1-m} \{x_n\},
$$
where each $x_n$ is inductively defined as $x_n:= M(x_{n+1-m},\ldots x_n)$. The $m$-linear operator $M$ is said to be \textit{hypercyclic} (in the B\`es and Conejero sense \cite{BesCon14}) if there exists an $m$-tuple $(x_{1-m},\dots,x_0)\in X^m$ such that the $m$-linear orbit of $M$ with initial condition $(x_{1-m},\dots,x_0)$ is dense in $X$. In this case $(x_{1-m},\dots,x_0)$ is called a hypercyclic vector for $M$.

A family of functions $\{f_n:n\in\zN\}$, $f_n:X\to Y$, is said to be universal provided that there exists $x\in X$ such that its orbit $\{f_n(x):n\in \zN\}$ is dense in $Y$. Also, the family is said to be transitive if for all nonempty open sets $U\sub X$ and $V\sub  Y$, there exists $n$ such that $f_n(U)\cap V\neq \emptyset$. Thus, if we define inductively 
\begin{align*}
M^1(x_{1-m},\dots,x_0) & =  M(x_{1-m},\dots,x_0) ,\\
& \vdots \\
M^{ m}(x_{1-m},\dots,x_0) & =  M(x_0,M^1(x_{1-m},\dots,x_0),\dots,M^{ m-1}(x_{1-m},\dots,x_0)) ,\\
M^{n}(x_{1-m},\dots,x_0) &  =   M(M^{ n-m}(x_{1-m},\dots,x_0),\dots,M^{ n-1}(x_{1-m},\dots,x_0)), \quad \text{for }n>m;
\end{align*}
 we have by definition that an $m$-linear operator $M$ is hypercyclic if and only if the family $\{M^{ n}:n\in\zN\},\ M^{ n}:\overbrace{X\times \ldots \times X}^m\to X$ is universal. Since the universal vectors of an universal family are always $G_\delta$ it follows that the set of hypercyclic vectors of a hypercyclic multilinear operator is a $G_\delta$ set.

 It is well known (see \cite{Gro99}) that if $X$ is a complete metric space and $Y$ is separable, a family is transitive if and only if it is universal and  the universal vectors are residual.  Therefore, if the family $\{M^{ n}\}$ is transitive then $M$ results hypercyclic with a residual set of hypercyclic vectors. This allows us to give a notion of transitivity for a hypercyclic $m$-linear operator $M$.
 
 \begin{definition}
 An  $m$-linear operator $M$ is said to be strongly transitive provided that the family $\{M^{ n}\}$ is transitive. Equivalently, $M$ is hypercyclic with a residual set of hypercyclic vectors.
 \end{definition} 

 Thus, Question  B can be formulated in the following way:\\
 \textit{
 Let $X$ be a Fr\'echet space and $M$ an $m$-linear hypercyclic operator. Is  $M$ necessarily strongly transitive?}

 Notice that if $X$ is a Banach space, then no $m$-linear operator $M$ can be strongly transitive. 
Indeed, in the same way as in the case of homogeneous polynomials \cite{Ber98,CardMurPoli} and of bihypercyclic operators \cite{GroKim13}, it is possible to define a notion of limit ball: if $x_{1-m},\ldots x_0 \in {\frac{1}{\|M\|^{m-1}}}B_X$ then the orbit
$Orb_M\{x_{1-m},\ldots,x_0\}$ is contained in ${\frac{1}{\|M\|^{m-1}}}B_X$, moreover the orbit is a sequence that converges to zero. Therefore the hypercyclic vectors can not be dense in $X^m$ and consequently the operator is not strongly transitive. 

In \cite{BesCon14} it was proved that the multilinear operator $M=e_1'\otimes\dots\otimes e_1'\otimes B$
is hypercyclic in $\C^\zN$, where $\C^\zN$ is the space of complex sequences with fundamental system of seminorms $\|a\|_k=\max_{j\leq k} |a_j|$ and $B$ is the backward shift.  At the same time it was proved that the set of hypercyclic vectors is residual. Therefore it follows that $M$ is strongly transitive. Here we follow a different approach that we believe is simpler. We prove directly that $M$ is strongly transitive.

\begin{proposition}\label{C a la N}
Let $M\in \LL(^{m} \C^\zN)$, $M(x_{1-m},\ldots,x_0)=[x_{1-m}]_1\ldots [x_{-1}]_1B(x_0)$. Then $M$ is strongly transitive.
\end{proposition}
\begin{proof}
The iterations of a vector $(x_{1-m},\ldots, x_0)$ are 
$$
M^n(x_{1-m},\ldots,x_0)=c_n(x_{1-m},\ldots,x_0)B^n(x_0),
$$
where $c_n(x_{1-m},\ldots, x_0)$ is a continuous function that depends on  $[x_{1-m}]_1,\ldots, [x_{-1}]_1$ and on the $n-1$ first coordinates of $x_0$. It follows, by an easy inductive argument, that the weights $c_n$ satisfy the recursive relation
$c_{m+j+1}=c_{j+1}\cdots c_{j+m}[x_0]_{j+2}\cdots [x_0]_{j+m}$.

Let $U_{1-m},\ldots U_0$ be nonempty sets. Since the family of sets $B_{\epsilon,k,x}=\{y:\|x-y\|_k< \epsilon\}$ is a  basis of open neighborhoods for the topology of $\C^\zN$, we may suppose that $B_{\epsilon,k,w}\sub U_0$ for some $k>m$, $w\in\C^\zN$. We will show that $M^{k}(U_{1-m},\ldots U_0)$ is $\C^\zN$. For $i< m$ let $x_{i-m}\in U_{i-m}$ such that $[x_{i-m}]_1\neq 0$ and $x_0\in B_{\epsilon,k,w}$ satisfying $[x_0]_j\neq 0$ for $1\leq j\leq k$.  Let $z\in \C^\zN$, let $S$ be the forward shift operator. We have that $x_0+\f{S^{k}(z)}{c_k(x_{1-m},\ldots, x_0)}\in U_0$ and since $c_n$ reads only the first coordinates of $x_{1-m},\ldots x_{-1}$ and the first $n-1$ coordinates of $x_0$, 
$$
M^k\left(x_{1-m},\ldots, x_1,x_0+\f{S^{k}(z)}{c_k(x_{1-m},\ldots,x_0)}\right)=c_k(x_{1-m},\ldots,x_0)B^k\l x_0+\f{S^{k}(z)}{c_k(x_{1-m},\ldots, x_0)}\r=z.
$$
\end{proof}
The space $\H$ of entire functions on the complex plane is, endowed with the compact open topology, a non normable Fr\'echet space. The continuous seminorms are $\|f\|_K= \sup_{z\in K}|f(z)|$, where $K\sub\C$ is a compact set. Thus, the sets $U_{\epsilon,f,R}=\{h\in \H:\|f-h\|_{B(0,R)}<\epsilon \}$ form a basis for the compact open topology.

  Adapting the techniques used in \cite{CardMurH(C)} to prove that the polynomial $P(f)=f(0)f(\cdot +1 )$ is hypercyclic in $\H$ we will prove that the bilinear operator $M(f,g)=f(0)g(\cdot+1)$ is strongly transitive in $\H$.


\begin{theorem}
 The bilinear operator $B\in \LL(^2\H)$ defined as $B(g,f)(z)=g(0)f(z+1)$ is strongly transitive.
\end{theorem}
\begin{proof}
 Let $U_1,U_2,V$ be nonempty open sets. We may suppose that
 \begin{align*}
  U_1&=\{h\in \H: \|h-f_1\|_{B(0,R)}<\epsilon\},\\
  U_2&=\{h\in \H: \|h-f_2\|_{B(0,R)}<\epsilon\},\\
  V&=\{h\in \H: \|h-g\|_{B(0,R)}<\epsilon\};\\
 \end{align*}
where $R>2$, $R$ is not a natural number and $g,f_1,f_2$ do not have zeros on the integer numbers. We will show that $B^{ 2n_0}(f_1,h)\in V$, for some $h\in U_2$, where $n_0=\lfloor R\rfloor +1$. Note that $R<n_0<2n_0-R<n_0+1$. Thus $n_0$ is the only natural number in $\{k\in\zN:R<k<2n_0-R\},$ and $[0,2n_0]\cap \zN \sub [0,R]\cup \{n_0\}\cup [2n_0-R,2n_0]$.

Observe that $B^{ n}(f_1,f_2)(z)=c_n(f_1,f_2) f_2(z+n)$, where
$$c_n(f_1,f_2)=f_1(0)^{F_n}f_2(0)^{F_{n-1}}...f_2(n-2)^{F_1}$$
and $F_n$ is the classical Fibonacci sequence
$$\begin{cases}
   F_1&=1;\\
   F_2&=1;\\
   F_{n}&=F_{n-1}+F_{n-2}.
  \end{cases}$$

Now consider for each $l\in \mathbb N$,
\begin{align*}
  U_2^l&=\{h\in \H: |h(z)-f_2(z)|<\f{\epsilon}{l} \text{ for every } z\in B(0,R)\}\\
  V^l&=\{h\in \H: |h(z)-g(z-2n_0)|<\f{\epsilon}{l} \text{ for every } z\in B(2n_0,R)\},\\
  W^l&=\{h\in \H: |h(z)-\alpha|<\f{\epsilon}{l} \text{ for every } z\in B(n_0,\delta)\},\\
\end{align*}
where $\delta$ is small enough such that $B(0,R)$, $B(2n_0,R)$, $B(n_0,\delta)$ are pairwise disjoint and $\f{1}{\alpha}$ is any $F_{n_0-1}$-th root of the number
$$f_1(0)^{F_{2n_0}}f_2(0)^{F_{2n_0-1}}f_2(1)^{F_{2n_0-2}}\ldots f_2(n_0-1)^{F_{n_0}}g(-n_0+1)^{F_{n_0-2}}\ldots g(-2)^{F_1}.$$

By Runge's Theorem there exists, for each $l$, a function $h_l\in U_2^l\cap V^l\cap W^l$. Thus, $\|h_l-f_2\|_{B(0,R)}\rightarrow 0,\ \|h_l-\tau_{-2n_0}g\|_{B(2n_0,R)}\rightarrow 0$ and $\|h_l-\alpha\|_{B(n_0,\delta)}\rightarrow 0$, as $l\to \infty$. Notice that,
by the choice of $\alpha$, $c_{2n_0}(f_1,h_l)\rightarrow 1$. Therefore we have,
\begin{align*}
\|c_{2n_0}(f_1,h_l)h_l-\tau_{-2n_0}g\|_{B(2n_0,R)}&\leq |c_{2n_0}(f_1,h_l)-1|\cdot\|h_l\|_{B(2n_0,R)}+\|h_l-\tau_{-2n_0}g\|_{B(2n_0,R)}\\
&\leq |c_{2n_0}(f_1,h_l)-1|\cdot\|h_l-\tau_{-2n_0}g\|_{B(2n_0,R)}\\
&\hspace{2cm}+|c_{2n_0}(f_1,h_l)-1|\cdot\|\tau_{-2n_0}g\|_{B(2n_0,R)}+\f{\epsilon}{l}\rightarrow 0.
\end{align*}
So, for large enough $l$, we have that 
$$\|c_{2n_0}(f_1,h_l)h_l-\tau_{-2n_0}g\|_{B(2n_0,R)}<{\epsilon}$$
or equivalently $B^{ 2n_0}(f_1,h_l)\in V$. Since $f_1\in U_1$ and $h_l\in U_2$, we conclude that $B$ is multilinear hypercyclic.
\end{proof}

In \cite{BesCon14} it was shown that the bilinear operator $M\in \LL^2(\H)$, $M(f,g)=f(0)g'$ is a hypercyclic operator (note that, in contrast, its associated homogeneous polynomial $f\mapsto f'(0)f$ is not hypercyclic, see \cite{{CardMurH(C)}}). Here we present a different proof of this fact that we believe is simpler. We also show that the operator is not strongly transitive and thus the set of hypercyclic vectors is not residual. This gives an answer for Question B.

\begin{theorem}\label{Hiper en H(C)}
Let $M(f,g)=f(0)g'(z)$. Then $M$ is hypercyclic and not strongly transitive. 
\end{theorem}
\begin{proof}

We start by computing the orbit of $(f,g)$. We have that $M^n{(f,g)}(z)=c_n(f,g) g^{(n)}(z)$, where
$$c_n(f,g)=f(0)^{F_n}g(0)^{F_{n-1}}\cdot\ldots \cdot g^{(n-2)}(0)^{F_1}$$
and $(F_n)_n$ is the classical Fibonacci sequence.
The weights $c_n$ satisfy also the recursive relations $$\begin{cases}
c_1(f,g)=f(0);\\
c_2(f,g)=f(0)g(0);\\
c_{n+1}(f,g)=c_n(f,g)c_{n-1}(f,g)g^{(n-2)}(0).
\end{cases}$$
 
 We will exhibit an universal vector of the type $(1,f)$. So, we will write $c_n(f)$ instead of $c_n(1,f)$. 
 The idea is to construct a function $Q$ such that for some sequence $(n_j)_j$, $M^{n_j}(1,Q)=Q^{(n_j)}$ (i.e. $c_{n_j}(Q)=1$) and such that   $Q^{(n_j)}-p_j\to 0$ for an appropriate dense sequence $\{p_j\}_j$ in $H(\mathbb C)$. 
 
 For $\lambda\in \C$ , let $\lambda^\f{1}{n}$ be the  $n$-root of $\lambda$ whose argument is $\f{arg(\lambda)}{n}$, so that $\lambda^\f{1}{n}\to 1$ for every $\lambda$.

{Given a polynomial $p(z)=\sum_{l=0}^n \f{z^l}{l!}$ we will consider its usual primitive $I(p):= \sum_{l=1}^{n+1} a_l\f{z^l}{l!}.$} 
 
Let $\{p_n\}_n$  be a dense sequence of polynomials such that $p_n(z)=\sum_{i=0}^n a_{i,n}z^i$ and for $0\leq i\leq n$ we have that $\f{1}{n}\leq |a_{i,n}| \leq n$ and that $a_{i,n}=0$ if $i>n$. 

We will construct our universal function inductively. The first step is simple. Set $n_1=3$. We define $Q_1$ as $\alpha_1+ a_{0,1}z+ a_{1,1} \frac{z^2}{2!}+  \beta_1\frac{z^3}{3!}$,
where $\alpha_1$ and $\beta_1$ are complex numbers such that the weights $c_4(Q_1)$ and $c_5(Q_1)$ are both one. Since $|a_{0,1}|=|a_{1,1}|=1$ it follows that $|\alpha_1|=|\beta_1|=1$. 

Step two: for a large number $n_2$ to be specified later we consider the unique complex numbers $\alpha_2$ and $\beta_2$ such that 
$Q_2=Q_1+\alpha_2 \frac{z^{4}}{4!}+I^{5} p_2 +\sum_{l=8}^{n_2-1}1\frac{z^l}{l!}+\beta_2\f{z^{n_2}}{n_2!}$ 
satisfy that both $c_{n_2+1}(Q_2)$ and $c_{n_2+2}(Q_2)$ are both one.
 We claim that if $C>0$ is such that 
	\begin{align}\label{cte paso inductivo}
     j^\frac{F_{k+1}-1}{F_k}\le C2^j \textrm{ for every }j,k,
	\end{align}
then $\alpha_2<C 2^{n_1+1}$ and, for sufficiently large $n_2$, $\beta_2<C 2^{n_2}$. To show this we notice that, since $c_5(Q_2)$ and $c_4(Q_2)$ are equal to one, then $M^5(1,Q_2)=Q_2^{(5)}=M(1,Q_2^{(4)})$ and $M^6(1,Q_2)=Q_2^{(4)}(0)Q_2^{(6)}=M^2(1,Q_2^{(4)})$, and thus $c_{n}(Q_2)=c_{n-4}(Q_2^{(4)})$ for every $n>0$. 
Therefore if we put $k=n_2-4$, then $c_{n_2}(Q_2)=c_{k}\l \alpha_2 +I(p_2) +\sum_{l=4}^{k-1}1\frac{z^l}{l!}+\beta_2\f{z^{k}}{k!}\r$.
Now $c_{k+1}(Q_2^{(4)})=c_{n_2-4+1}(Q_2^{(4)})=c_{n_2+1}(Q_2)=1$ and therefore
$$\alpha_2^{F_k}\cdot a_{0,2}^{F_{k-1}}\cdot\ldots\cdot a_{2,2}^{F_{k-3}}\cdot 1 \cdot\ldots \cdot 1=1.$$
If we define $\Gamma_{2,k}$ as the number  $|a_{0,2}|^{F_{k-1}}\cdot\ldots\cdot |a_{2,2}|^{F_{k-3}}$ this implies that
$\alpha_2= \Gamma_{2,k}^{-\frac{1}{F_k}}.$ Using that $\frac12\le |a_{i,2}|\leq 2$ and that $\sum_{l=1}^N F_{j}=F_{N+2}-1$, we have that $|\alpha_2|\leq \prod_{l=1}^{k-1} 2^\frac{F_l}{F_k}=2^\frac{F_{k+1}-1}{F_k}\le C 2^2 =C 2^{n_1+1}$.
            
Now we look at the condition $c_{k+2}(Q_2^{(4)})=1$ to obtain
$$\alpha_2^{F_{k+1}}\cdot a_{0,2}^{F_k}\cdot\ldots a_{2,2}^{F_{k-2}}\cdot 1\ldots \cdot 1 \cdot \beta_2=1.$$
Hence
$$|\beta_2|={\Gamma_{2,k}}^\frac{F_{k+1}}{F_k}\cdot \frac{1}{\Gamma_{2,k+1}}.$$
Now we compute this number using the Vajda's identity, \cite[Appendix (20a)]{Vaj08} 
$$F_{N+i}F_{N+j}-F_NF_{N+i+j}=(-1)^N F_{i}F_j.$$
Applying the above formula for each $N=k-l-1$, $i=1$ and $j=l+1$ we get
\begin{align*}
|\beta_2|&={\Gamma_{2,k}}^\frac{F_{k+1}}{F_k}\cdot \frac{1}{\Gamma_{2,k+1}}=\prod_{l=0}^2 a_{0,l}^{F_{k-l-1}\frac{F_{k+1}}{F_k}-F_{k-l}}\\
&=\prod_{l=0}^2 a_{0,l}^{\frac{(-1)^{k-l}F_{l+1}}{F_k}}\to 1 \text { as } k\to \infty.
\end{align*}
Therefore if $n_2$ is sufficiently large, $|\beta_2|\leq C 2^{n_2}$.

Step three (inductive step): suppose that we have defined $Q_1,\ldots Q_k$, $\alpha_1,\ldots \alpha_k$, $\beta_1,\ldots,\beta_k$, and numbers $n_1,\ldots, n_k$ such that for each $1\le j< k$,
\begin{enumerate}
\item $Q_{j+1}=Q_j+\alpha_{j+1}\frac{z^{n_{j}+1}}{(n_{j}+1)!}+I^{n_j+2}(P_{j+1})
  +\sum_{l=n_j+j+4}^{n_3-1} 1\cdot \frac{z^l}{l!}+ 
  \beta_{n_{j+1}}\frac{z^{n_{j+1}}}{n_{j+1}!}$;
\item $c_{n_{j+1}+1}(Q_{j+1})=c_{n_{j+1}+2}(Q_{j+1})=1$;
\item If $C$ is as in \eqref{cte paso inductivo} the same constant as in the previous step, $|\alpha_{j+1}|\leq C 2^{n_{j}+1}$ and
 $|\beta_{j+1}|\leq C 2^{n_{j+1}}$. 
\end{enumerate}
 
The construction of $Q_{k+1}$ satisfying conditions 1, 2 and 3 is achieved following exactly as in the second step, so we omit the details. 

Let $Q=\lim_k Q_k$. It is well defined because for each $l\le k$ we have that $|Q_k^{(l)}|\leq C2^l$. 

To show that $Q$ is universal we will show that $c_{n_k+2}(Q)Q^{(n_k+2)}-p_k\to 0$. We  consider the fundamental system of continuous seminorms given by 
 \begin{equation}\label{seminormas H(C)}
\|f\|_k= \sup_{j\geq 0} |a_n| \f{k^j}{j!},
 \end{equation}
  where, $k\in\mathbb N$ and $f(z)=\sum_{n=0}^\infty a_nz^n$. This seminorms generate the usual topology on $H(\mathbb C)$
(see for example \cite[Example 27.27]{MeiVog97} or \cite{Per99}).

By Condition 2, $c_{n_k+2}(Q)=1$. Since the first $k$ derivatives evaluated at zero of $Q^{(n_k)+2}$ and $p_k$ are equal, and the derivatives of $Q^{(n_k)+2}$ at zero between $k+1$ and $n_{k+1}-1$ are equal to one, we have that for each $n$, 
\begin{align*}
\|c_{n_k+2}(Q)Q^{(n_k+2)}-p_k\|_n&=\sup_{j>k} \left|Q^{(n_k+2)+j}(0)\right| \frac{n^j}{j!}
\leq \max\left\{\frac{n^k}{k!}, \sup_{j\geq n_{k+1}}C 2^{n_k+j+2}\frac{n^j}{j!}\right\}\to 0 \quad \textrm{as }k\to\infty.\\
\end{align*} 

To show that the bilinear operator is not strongly hypercyclic it suffices to prove that the set of hypercyclic vectors is not dense.
We claim that that there is some $\delta>0$ such that if $|f(0)|<\delta$ and if $\|g\|_1<1$, then $M^n(f,g)\to 0$. Therefore, $(f,g)$ is not a hypercyclic vector.

Let $k\geq 1$ such that $k2^{2^\f{n}{2}}\geq (n-2)!2^{2^\f{n-1}{2}}$ for every $n$ and let $\delta=\f{1}{k2^{2}}$. Since $\|g\|_1<1$ we have that $|g^{(n)}(0)|\leq n!$ for every $n$. We claim that $|c_n(f,g)|\leq \f{1}{k2^{2^\f{n}{2}}}$ for every $n$. Indeed, for $n=1$ we have that $|c_1(f,g)|=|f(0)|<\delta<\f{1}{k2^{2^\f{1}{2}}}$ and for $n=2$ we have that $|c_2(f,g)|=|f(0)g(0)|<\delta\cdot 1=\f{1}{k2^{2}}$. Suppose that our claim is true for $n\geq 2$. Then
\begin{align*}
c_{n+1}(f,g)&=c_n(f,g)c_{n-1}(f,g) g^{(n-2)}(0)\leq \f{(n-2)!}{k2^{2^\f{n}{2}} k2^{2^\f{n-1}{2}}}\\
&\leq \f{1}{k^2(2^{2^\f{n-1}{2}})^2}\leq \f{1}{ k2^{2^\f{n+1}{2}}}. 
\end{align*}

Applying the Cauchy inequalities we obtain that

\begin{align*}
\|M^n(f,g)\|_R&=\|c_n(f,g)g^{(n)}\|_R\leq \f{1}{k2^{2^\f{n}{2}}} \|g^{(n)}\|_R\leq \f{n!}{R^{n+1}2\pi k2^{2^\f{n}{2}}} \|g\|_{2R}\to 0.
\end{align*}
\end{proof}

We finalize the section with a last example in $\H$. We will show that the bilinear operator $N(f,g)=g(0)\cdot f'$ is hypercyclic. The dynamics induced by this operator and its transpose $M(f,g)=f(0)g'$ are quite different. Indeed, while in $M$ only $g$ and the number $f(0)$ determine the orbit of $(f,g)$, in $N$ both $f$ and $g$ are relevant.  

\begin{example}\label{delta D}
Let $N\in\LL(^2\H)$ be the bilinear operator $N(f,g)=g(0)D(f)$, where $D$ is the derivation operator. Then $N$ is hypercyclic and not strongly transitive.
\end{example}
\begin{proof}
The orbit with initial conditions $(f,g)$ is
 \begin{equation}\label{OrbDerivar}
 N^{ n}(f,g)=\begin{cases}
                   c_{n}(f,g) D^\frac{n+1}{2}(f) & \text{ if } n \text{ is odd;}\\
                   c_{n}(f,g)D^\frac{n}{2}(g) & \text{ if } n \text{ is even.}
                   \end{cases}
 \end{equation}
Where $c_n(f,g)$ is defined as 
$$c_n(f,g)=\begin{cases}
            g(0)^{F_n}f'(0)^{F_{n-1}}g'(0)^{F_{n-2}}f^{(2)}(0)^{F_{n-3}}\ldots g^{(\f{n-1}{2})}(0)^{F_1} & \text{ if } n \text{ is odd};\\
            g(0)^{F_n}f'(0)^{F_{n-1}}g'(0)^{F_{n-2}}f^{(2)}(0)^{F_{n-3}}\ldots g^{(\f{n-2}{2})}(0)^{F_2}  f^{(\f{n}{2})}(0)^{F_1} & \text{ if } n \text{ is even,}\\
           \end{cases}$$
 and the $F_n$ are the usual Fibonacci numbers.          
The weights $c_n(f,g)$ may be seen in the following way. Consider $h(z)$ the entire function $h(z)=\sum_{j=0}^\infty \frac{g^{(j)}(0)}{(2j)!}z^{2j}+\sum_{j=1}^\infty \frac{f^{(j)}(0)}{(2j-1)!}z^{2j-1}$. Thus, $h$ is a merge between $g$ and $f$, satisfying  $h^{(n)}(0)=g^{(\frac{n}{2})}(0)$ if $n$ is even and $h^{(n)}(0)=f^{(\frac{n+1}{2})}(0)$ if $n$ is odd.
If we define $c_n:\H\to\C$ as 
$c_n(\tilde f)= 
\tilde f (0)^{F_n}\cdot\ldots  \tilde f^{(n-1)}(0)^{F_1}$, then $c_n(h)=c_n(g,f)$. Since $c_n(\tilde f)$ satisfies the recurrence relation $c_{n+1}(\tilde f)=c_n(\tilde f)\cdot c_{n-1}(\tilde f)\cdot \tilde f^{(n)}(0),$  $c_n(f,g)$ satisfies
$$
c_{n+1}(f,g)=\begin{cases}
c_n(f,g)\cdot c_{n-1}(f,g)\cdot g^{(\f{n-1}{2})}(0)&\text{ if } n \text{ is odd;}\\
c_n(f,g)\cdot c_{n-1}(f,g)\cdot f^{(\f{n}{2})}(0)&\text{ if } n \text{ is even.}
\end{cases}$$
           
We focus our attention only in the even iterations and forget the odd ones. If we construct $(f,g)$ such that the even iterations are dense, then the whole orbit will be dense. We rewrite the even iterations as $D^n(g) c_{2n}(f,g)$. Notice that this is a universal operator multiplied by certain weights depending on both $f$ and $g$. Therefore, if we find a $D$-hypercyclic vector $g$ and a function $f\in\H$ so that $c_{2n}(f,g)=1$, then the orbit will be dense. 

Suppose now that $g\in\H$ is fixed and $g\sim \sum a_n \f{z^n}{n!}$. We claim that $f\sim \sum b_n \f{z^n}{n!}$, where $b_n= \prod_{i< n} a_i^{-1}$, satisfies $c_{2n}(f,g)=1$.

Indeed, we use the well known identity 
\begin{equation}\label{eqfib1}
 F_{2n}=\sum_{j=1}^{n} F_{2j-1}.
\end{equation}
Hence,
\begin{align*}
c_{2n}(f,g)&=a_0^{F_{2n}}b_1^{F_{2n-1}}a_1^ {F_{2n-2}}\ldots a_{n-1}^{F_1}b_n^{F_1}\\
&=a_0^{F_{2n}}(a_0^{-1})^{F_{2n-1}}a_1^ {F_{2n-2}} (a_0^{-1}a_1^{-1})^F_{2n-3}\ldots a_{n-1}^{F_1}(a_0^{-1}\ldots a_{n-1}^{-1})^{F_1}\\
&=a_0^{F_{2n}-\sum_{j\leq n} F_{2j-1}}\cdot \ldots \cdot a_{n-1}^{F_2-F_1}=1.
\end{align*}

The proof finalize by constructing a $D$-hypercyclic function $g$ such that its associated function $f$ is well defined. The $D$-hypercyclic function $g$ can be constructed as follows. Let $\{P_n\}_n$ be a dense sequence of polynomials that satisfy that $deg(P_n)=n$,  $P_n=\sum_{j=1}^{n} \alpha_{n,j} \f{z^j}{j!}$, with $\alpha_{n,j}\neq 0$ and $n>\alpha_{n,j}>\f{1}{n}$ for $j\le n$. Let $k_1=0$ and $k_n=\sum_{j<n} j$ for $n>1$.  
For a polynomial $P(z)=\sum_{j=0}^n \f{a_j}{j!}z^n$ we consider its usual primitive $I(p)=\sum_{j=1}^{n+1} \f{a_j}{j!}z^j.$
We claim that $g=\sum_{n} I^{k_n}(P_n)$ is hypercyclic. It is easy to see that $g$ is a well defined entire function. 
To prove
that it is hypercyclic, we will use the seminorms defined in \eqref{seminormas H(C)}. Then
\begin{align*}
\|D^{k_n}g-P_n\|_k&=\|\sum_{j=n+1} I^{k_j-k_n} P_{j}\|_k\\
&\leq \sum_{j=n+1} \|I^{k_j-k_n} P_{j}\|_k\\
&<\sum_{j=n+1} j\f{k^j}{j!}\to 0.
\end{align*}
Finally we prove that $f$ is well defined. Observe that if $g(z)=\sum_na_nz^n$, then  $a_n=\alpha_{l,j}$ for $n=k_l+j$.  
Therefore, $|a_n|>\frac1{l}>\frac{1}{(2n)^{\f{1}{2}}}$.
This implies that for each $k\in\mathbb N$,
\begin{align*}
\|f\|_k&=\sup_j \l\prod_{i< j} a_i^{-1}\r \f{k^j}{j!}< \sup_j 2^{\f{1}{2}}j!^{\f{1}{2}}\f{k^j}{j!}<\infty.
\end{align*}

To show that $N$ is not strongly hypercyclic it suffices to show that the set of hypercyclic vectors is not dense in $\H\times \H$. In the proof of Theorem \ref{Hiper en H(C)} we showed that there is some $\delta>0$ and $k>1$ such that for every $f,g$ such that $|f(0)|<\delta$ and $\|g\|_1<1$, we have $\tilde c_n(f,g):= \left |f(0)^{F_n}g(0)^{F_{n-1}}\ldots g^{(n-2)}(0)\right |\leq \f{1}{k2^{2^\f{n}{2}}}$. Thus, if $h(z)=\sum_{j=0}^\infty \frac{g^{(j)}(0)}{(2j)!}z^{2j}+\sum_{j=1}^\infty \frac{f^{(j)}(0)}{(2j-1)!}z^{2j-1}$ we have that if $|h(0)|<\delta$ and $\|h'\|_1<1$, then for every $n\in \mathbb N$, $|c_n(f,g)|=\tilde c_n(h,h')<\f{1}{k2^{2^\f{n}{2}}}$.
In this case we obtain, by the Cauchy inequalities and by \eqref{OrbDerivar}, that
\begin{align*}
\|N^n(f,g)\|_R\leq \f{n!\max\{\|f\|_{R+1},\|g\|_{R+1}\}}{k2^{2^\f{n}{2}}}\to 0.
\end{align*}
Since the set of pairs $(f,g)\in \H\times \H$ such that the function $h$ (defined as above) satisfies  $|h(0)|<\delta$ and $\|h'\|_1<1$ is open in $\H\times \H$, it follows that the set of hypercyclic vectors is not residual.
\end{proof}    
\section{Multilinear hypercyclic operators on arbitrary Banach spaces}\label{SecBan}

In this section we prove our main result,  Theorem \ref{main Theorem}. It establishes that hypercyclic multilinear operators may be found in arbitrary separable and infinite dimensional Banach spaces,  giving a positive answer to Question C. This implies that  there are hypercyclic multilinear operators that are not strongly transitive and whose hypercyclic vectors are not residual.

\subsection{ A first example in $\ell_p$}

Since the backward shift on $\ell_p$ operates like the differentiation operator on $\H$, from the results on the previous section the bilinear mapping
 $M(x,y)=e_1'(y)B(x)$ is a good candidate to be hypercyclic on $\ell_p$, $1\leq p<\infty$. This is indeed the case.
 \begin{example}\label{Ejemplo en c_0}
 Let $X=\ell_p$ or $c_0$, $p< \infty$ and $M(x,y)=e_1'(y)B(x)$. Then $M$ is hypercyclic.
\end{example}

Given vectors $x,y\in X$, the iterations $M^n(x,y)$ are
$$M^{ n}(x,y)=\begin{cases}
c_{n}(x,y)B^\frac{n+1}{2}(x) & \text{ if } n \text{ is odd;}\\
 c_{n}(x,y)B^\frac{n}{2}(y)& \text{ if } n \text{ is even.}
\end{cases}$$
Where $c_n(x,y)$ is defined as 
$$c_n(x,y)=\begin{cases}
y_1^{F_n}x_2^{F_{n-1}}y_2^{F_{n-2}}x_3^{F_{n-3}}\ldots y_{\f{n-1}{2}}^{F_1} & \text{ if } n \text{ is odd};\\
y_1^{F_n}x_2^{F_{n-1}}y_2^{F_{n-2}}x_3^{F_{n-3}}\ldots y_{\f{n-2}{2}}^{F_1} x_{\f{n}{2}}^{F_1} & \text{ if } n \text{ is even}\\
\end{cases}$$
and the $F_n$ are the usual Fibonacci numbers.

If we prove that the even iterations are dense, then the whole orbit will be dense. 
Note that if we are able to construct a hypercyclic vector $y$ for $2B$ and a well-defined vector $x$ such that $c_{2n}(x,y)=2^n$ then the even iterations $c_{2n}(x,y)B^n(y)$ form a dense sequence.

   To construct this vectors we will need the following lemma.

\begin{lemma}\label{construccion 1}
Let $X=\ell_p$ or $c_0$, $p< \infty$. 
Let $(a_n)_n\subset \mathbb C$ be the sequence such that $c_{2n}((2^{a_k})_k,1)=2^n$.
Then there exists a hypercyclic vector $y$ for $2B$ such that the vector $$z_{i+1}:=2^{a_i}\cdot\prod_{j\leq i}{y_j^{-1}}\in X.$$ 
\end{lemma}
 The construction of the universal vector is actually a simplified version of what is done in the next subsection, so we omit its proof here and refer the reader to Theorem \ref{construccion 2} (see also Lemma \ref{Lema a_n 1}).

\begin{proof}[Proof of Example \ref{Ejemplo en c_0}]
	If we define, like in Example \ref{delta D}, $x_{i+1}=\prod_{j\leq i}{y_j^{-1}}$, it follows that $c_{2n}(x,y)=1$. Thus, if $z$ is the sequence $(x_k2^{a_k})_k$, then $c_{2n}(z,y)=c_{2n}(x,y)c_{2n}((2^{a_k})_k,1)$. Therefore it suffices to find a vector $y$ and a sequence $(a_k)_k$ such that $c_{2n}((2^{a_k})_k,1)=2^n$, such that $y$ is a hypercyclic vector for $2B$ and such that $z_{i+1}=2^{a_i}x_{i+1}=2^{a_i}\prod_{j\leq i}{y_j^{-1}}$ defines a vector in $X$.
	The existence of the vector $y$ as needed is guaranteed by the previous lemma.
\end{proof}

\subsection{Hypercyclic bilinear operators on arbitrary Banach spaces}

We prove in this section our main result.
\begin{theorem}\label{main Theorem}
	Let $X$ be a separable infinite dimensional Banach space. Then for every $m\geq 1$ there exist an $m$-multilinear hypercyclic operator acting on $X$.  
\end{theorem}

The case $m=1$ is due to Ansari-Bernal \cite{ansari1997existence,bernal1999hypercyclic}, who  independently proved the existence of hypercyclic operators on arbitrary infinite dimensional separable Banach spaces. Later on Bonet and Peris \cite{BonePer98} generalized the result to Fr\'echet spaces. We will prove Theorem \ref{main Theorem} only for $m=2$, being the case $m>2$ analogous but the notation more technical.

The proof of Theorem \ref{main Theorem} will be divided in three different steps. First we will prove that certain weighted multilinear operator $M$ defined over $\ell_1$ is hypercyclic. Afterwards we will give a notion of quasiconjugation for multilinear operators, and prove that hypercyclity is preserved under quasiconjugation. This definition coincides with the one given by Grosse-Erdmann and Kim for bihypercyclity \cite{GroKim13}. Finally we will show that each separable infinite dimensional Banach space supports a multilinear operator quasiconjugated to $M$.  

The multilinear operator we are looking for is a generalization of Example \ref{delta D} to the Banach space $\ell_1$. Effectively this bilinear operator is hypercyclic. However, since we are looking for quasiconjugation in arbitrary Banach spaces, we need to weight the backwardshift. At the same time, the weights can not tend to zero too fast since we would loose hypercyclicity.
  
\subsubsection{First step} 
   
\begin{theorem}\label{En Banach}
Let $\omega=(1,\f{1}{2^2},\frac{1}{3^2},\frac{1}{4^2},\ldots)$ and $B_\omega: \ell_1\to \ell_1$ defined as $[B_\omega(x)]_i=\omega_i x_{i+1}$. The multilinear operator 
$M(x,y)=e_1'(y)B_\omega(x)$ is hypercyclic.
\end{theorem}
We first show the existence of a universal vector for a family of weighted shifts.
\begin{theorem}\label{construccion 2}
Let $a_n=1-\frac{n(n-1)}{2}$ and $B_\omega$ the weighted backward shift over $\ell_1$ with weights $\omega_n=\f{1}{n^2}$.
There exists an universal vector $y\in\ell_1$ for $2^{n}n!^2B_\omega^n$ so that the vector $x_{i+1}=2^{a_i}i!^2i^2
\prod_{j\leq i}y_{j}^{-1}$ is well defined in $\ell_1$.
\end{theorem}
\begin{proof}
 For an element $y\in \ell_1$ we consider the associated vector 
$$
[\Phi(y)]_i=2^{a_i} i^2i!^2\prod_{l\leq i} |[y]_l|^{-1}.$$
We need to construct an universal vector $y$ so that $\Phi(y)\in\ell_1$.
Let $(z_n)_n$ be a dense sequence in $c_{00}$, such that for all $n$, $n=\max \{i: [z_n]_i\neq 0\}$, and for all $i\le n$, $\f{1}{n}\leq |[z_n]_i|\leq n$.

Consider $S_\omega$ the formal right inverse of $B_\omega$. We will construct the vector proceedings as follows.
An usual universal vector for the family $\{2^{n}n!^2 B_\omega^n\}$ is of the form
$$\tilde z=\sum_k S_\omega^{n_k} \f{z_k}{2^{n_k}n_k!^2}.$$
Note however that this vector has gaps of zeros of length $n_k-k$,  thus it is impossible that $\Phi(\tilde z)\in \ell_1$. Therefore we will add a control vector in each gap, and enlarge the length of the gaps ($n_k-k$) to force $[\phi(z)]_i\leq \frac{1}{i^2}$ for all $i\leq k$. Let $n_1=0$.
Our universal vector will be of the type 
\begin{align*}
z&=z_1+\sum_{l=1+1}^{n_2} \f{1}{l^2}e_l+ S_\omega^{n_2} \frac{z_2}{2^{n_2}n_2!^2}+\sum_{l=n_2+2+1}^{n_3} \f{1}{l^22^{n_2}}e_l+S_\omega^{n_3} \f{z_3}{2^{n_3}n_3!^2}+\ldots\\
&=
\sum_{k=1}^\infty S_\omega^{n_k} \f{z_k}{2^{n_k}n_k!^2}+\sum_{k=2}^\infty \sum_{l=n_{k-1}+k-1+1}^{n_k} \f{1}{l^22^{n_{k-1}}}e_l.
\end{align*}
Notice that $z$ can be written as a limit of the vectors
$$
u_j=\sum_{k=1}^j S_\omega^{n_k} \f{z_k}{2^{n_k}n_k!^2}+
\sum_{k=2}^{j} \sum_{l=n_{k-1}+k-1+1}^{n_k} \f{1}{l^22^{n_{k-1}}}e_l.
$$
 Each $u_j$ extends $u_{j-1}$ and $[\Phi(u_j)]_i$ reads only the first coordinates. Hence we may construct the vectors $u_j$ inductively so that they  satisfy  $[\Phi(u_j)]_i\leq \frac{1}{i^2}$ for all $n_2<i<n_j+j+1$. This condition will pass automatically to $z$.

We now construct the numbers $n_k$ (i.e. the gaps), and hence $z$, by induction. Take $n_1=0$ and $u_1=z_1$. 
Let $n_2$ be such that, for $n\ge n_2$,

\begin{equation}
\label{n_2}
2^{\frac{n^2}{4}}n^4\cdot n!^4\cdot 2^{2n}\cdot 2^2\cdot2^{a_n}\leq 1
\end{equation}  \
This may be done, because $a_n\sim -\frac{n^2}{2}.$ We define $u_2$ as 
$$
u_2:=u_1+\sum_{l=1+1}^{n_2} \f{1}{l^2} e_l+ \f{S_\omega^{n_2}{z_2}}{2^{n_2}n_2!^2}.
$$

We claim that 
 $[\Phi(u_2)]_i\leq \f{1}{i^2}$ for $i=n_2+1,i=n_2+2$ but also that $[\Phi(\tilde u_2)]_i\leq \f{1}{i^2}$ for $i\geq n_2+1$, where $\tilde u_2=u_2+\sum_{l=n_2+2+1}^\infty\frac{1}{l^22^{n_2}}e_l$. Having a bound for  $[\Phi(\tilde u_2)]_i$ will help us to estimate $[\phi(u_3)]_i$ in the next step.
We need first to bound $\prod_{i=1}^l |[u_2]_i|^{-1}$ for $n_2+1\leq l\leq n_2+2$.

Recall that $|[z_2]_i|\geq \frac{1}{2}$ for $i\leq 2$, and that $supp(z_2)=[1,2]$. This gives us that $\left|[S_\omega^{n_2}(\frac{z_2}{2^{n_2}n_2!^2\\
})]_l\right|^{-1}\leq 2^{n_2}\cdot 2$ for $l=n_2+1$ and $l=n_2+2$.
By a direct estimation we get,
\begin{align}\label{x_2}
\nonumber \prod_{i=1}^l |[u_2]|^{-1}&=\prod_{i=1}^1 |[u_2]_i|^{-1}\cdot \prod_{i=2}^{n_2} |[u_2]_i|^{-1}\cdot
\prod_{i=n_2+1}^l |[u_2]_i|^{-1}\\
&\leq 1\cdot n_2!^2\cdot 2^{2n_2}\cdot 2^{2}.
\end{align}
Therefore, by \eqref{n_2} we get for $n_2+1\leq l\leq n_2+2$ that
\begin{equation*}
[\Phi(u_2)]_l\leq 2^{a_l} l!^2 l^2 \prod_{i=1}^l |[u_2]_i|^{-1}\leq l!^4\cdot l^2 \cdot 2^{2n_2}\cdot 2^2\cdot 2^{a_{l}}\leq \f{1}{l^2}.
\end{equation*}   
 
With this choice of $n_2$ we also get  that
$[\Phi(\tilde u_2)]_i\leq \frac{1}{i^2}$ for all $i\geq n_2+1$. To see this we need first get a  bound for $\prod_{l=1}^i
|[\tilde u_2]_l|^{-1}$, for $i\geq n_2+3$. Using inequality \eqref{x_2} we obtain,
 \begin{align*}\label{phidif}
 \nonumber \prod_{l=1}^i
 |[\tilde u_2]_l|^{-1}&\leq n_2!^2\cdot 2^{2n_2}\cdot 2^{2}\cdot\prod_{l=n_2+2+1}^i {l^22^{n_2}}\\
  &\leq i!^2\cdot 2^{2n_2}\cdot 2^{2}\cdot 2^{n_2(i-2-n_2)}\\
  & \leq i!^2\cdot 2^{2n_2}\cdot 2^{2}\cdot 2^{\frac{i^2}{4}}\cdot
  \leq i!^2\cdot 2^{2i}\cdot 2^{2}\cdot 2^{\frac{i^2}{4}}, 
 \end{align*}
 because
 $n_2(i-2-n_2)\leq \f{i^2}{4}$ for all $i\geq n_2$.
 Consequently, by the choice of $n_2$ and because $i\geq n_2$, 
 \begin{equation*}
 \Phi[(\tilde u_2)]_i \leq 2^{a_i} i!^2\cdot i^2\cdot  i!^2\cdot 2^{2i}\cdot 2^{2}\cdot 2^{\frac{i^2}{4}} \leq \frac{1}{(i+2)^2}\leq \frac{1}{i^2}.
 \end{equation*}

We will define inductively numbers $(n_j)_j$ (or the gaps $[n_{j-1}+j-1+1,n_j]$) and vectors $(u_j)_j$ and $(\tilde u_j)_j$ as we did in the first step. That is, $u_j$ extends $u_{j-1}$ and $\tilde u_j$ extends $u_j$. We want also $\tilde u_j$ to satisfy $[\Phi(\tilde u_j)]_i\leq \frac{1}{i^2}$ for all $n_2<i$. 

 Thus we define for all  $j\ge 2$,
 
\begin{enumerate}[label=(\roman*)]
\item  $$u_j=u_{j-1}+\sum_{l=n_{j-1}+j-1+1}^{n_j} \f{1}{2^{n_{j-1}}l^2}e_l+ S^{n_j}_\omega \l\f{z_{j}}{2^{n_j}n_j!^2}\r;$$
\item $$\tilde u_j=u_j+\sum_{l=n_j+j+1}^\infty \f{1}{2^{n_j}l^2} e_l;$$
\item \label{hipers2}  $\displaystyle \f{2^{n_{j-1}}n_{j-1}!^2 (n_j+j)^{2j}j^2}{2^{n_j}}<\frac{1}{j^2}$;
\item \label{n_js2}  for all $n\geq n_j$
$$(n+j)^4\cdot2^{\frac{n^2}{4}}\cdot n!^4\cdot2^{a_n}\cdot \pi_{j-1}\cdot 2^{n_{j-1}n}\cdot 2^{jn} \cdot j^j\leq 1,$$
where  $$\pi_{j-1}:=\prod_{l=1}^{n_{j-1}+j-1} |[u_{j-1}]_l|^{-1}.$$  
\end{enumerate}

The sequences $(n_j)_j$, $(u_j)_j$ and $(\tilde u_j)_j$ are well defined because conditions (i), (ii) and (iii) are automatically fulfilled  taking $n_j$  large enough.


%

We claim that each $u_j$ satisfies that for $n_2<i\leq n_j+j$,
\begin{equation}\label{phis2}
|[\Phi(u_j)]_i|\leq \frac{1}{i^2}.
\end{equation}
Since each $u_j$ extends $u_{j-1}$ and for any $v$, $[\phi(v)]_i$ depends only on the first $i$-coordinates of $v$, it is enough to prove that $[\phi(u_j)]_i\leq \frac{1}{i^2}$ for $n_{j-1}+j-1< i\leq n_j+j$. If $n_j+1\leq i\leq n_j+j$ it follows, by a direct estimation, that  
\begin{equation*}
 \prod_{l=1}^i
  |[u_j]_l|^{-1}\leq \pi_{j-1}\cdot\left(\prod_{l=n_{j-1}+j}^{n_j} 2^{n_{j-1}}l^2\right)\cdot \left(\prod_{l=n_{j}+1}^{i} 2^{n_{j}}n_j!^2 \left|[S_\omega^{n_j}(z_j)]_l\right|^{-1}\right)     \leq \pi_{j-1}\cdot n_j!^2\cdot 2^{n_{j-1}n_j}\cdot 2^{jn_j} \cdot j^j.
 \end{equation*} 
Applying \ref{n_js2} it follows that $[\Phi(u_j)]_i\leq \f{1}{i^2}$.\\

  If $n_{j-1}+j-1< i\leq n_j$, we have that $n_{j-1}(i-{n_{j-1}-j+1})\leq \frac{i^2}{4}$ and hence
\begin{align*}
\prod_{l=1}^i [u_j]_l^{-1}&\leq \pi_{j-1}\cdot \frac{i!^2}{(n_{j-1}+j-1)!^2} 2^{n_{j-1}(i-{n_{j-1}-j+1})}\\
&\leq \pi_{j-2}\cdot\left(\prod_{l=n_{j-2}+j-1}^{n_{j-1}} 2^{n_{j-2}}l^2\right)\cdot \left(\prod_{l=n_{j-1}+1}^{n_{j-1}+j-1} 2^{n_{j-1}}n_{j-1}!^2 \frac{l!^2(j-1)}{(l+n_{j-1})!^2} \right)\cdot \frac{i!^2}{(n_{j-1}+j-1)!^2} 2^{\f{i^2}{4}}\\
& \le  \pi_{j-2}\cdot n_{j-1}!^2\cdot 2^{n_{j-2}n_{j-1}}\cdot 2^{(j-1)n_{j-1}} \cdot (j-1)^{j-1}\cdot \frac{i!^2}{(n_{j-1}+j-1)!^2} 2^{\f{i^2}{4}}.
\end{align*} 
Therefore, by \ref{n_js2} applied to $n_{j-1}$,
\begin{align*}
[\Phi(u_j)]_i=  2^{a_i}i!^2i^2 \prod_{l=1}^i [u_j]_l^{-1}\leq 2^{a_i}i!^2i^2  \pi_{j-2}\cdot  2^{n_{j-2}n_{j-1}}\cdot 2^{(j-1)n_{j-1}} \cdot (j-1)^{j-1}\cdot  2^{\f{i^2}{4}} \le \frac{1}{i^2}.
\end{align*}
The universal vector we are looking for is
$z=\lim_{j\to \infty} u_j$. To see that the vector is well defined we  use inequality \ref{hipers2}. For $n_j+1 \leq l\leq j+n_j$ we have that,
$$
|[z]_l|=\left|\left[S_\omega^{n_j} \f{z_j}{2^{n_j}n_j!^2}\right]_l\right|\leq \left|\f{j(n_j+j)!^2}{2^{n_j}n_j!^2}\right|\leq \frac{j(n_j+j)^{2j}}{2^{n_j}}.
$$
Hence,
\begin{align}\label{norma}
 \left\|S_\omega^{n_j} \f{z_j}{2^{n_j}n_j!^2}\right\|_1&\leq
\frac{j^2(n_j+j)^{2j}}{2^{n_j}}\\
\nonumber&\leq\f{1}{j^2}.
\end{align} 
 Thus we obtain,
\begin{align*}
    \|z\|_1=\left\|\sum_{j=1}^\infty S_\omega^{n_j} \f{z_j}{2^{n_j}n_j!^2} \right\|_1 + \left\|\sum_{j=1}^\infty \sum_{k=n_{j}+j+1}^{n_{j+1}} \f{e_k}{k^2 2^{n_{j}}}\right\|_1\le 2\left\|(\frac1{j^2})_j\right\|_1<\infty.
\end{align*}
It remains to see that $\Phi(z)$ is well defined and that $z$ is universal. The well definition of $\Phi(z)$ is deduced from \eqref{phis2}. Indeed, $z$ extends $u_j$ for each $j$ and if $n_j>i$ we have $|[\Phi(z)]_i|=|[\Phi(u_j)]_i|\leq \frac{1}{i^2}.$

Finally the vector results universal by inequalities \ref{hipers2} and \eqref{norma}. It suffices to show that $2^{n_j}n_j!^2 B_\omega^{n_j} (z)-z_j\to 0$, 
\begin{align*}
    \|2^{n_j}n_j!^2B_\omega^{n_j} (z)-z_j\|&=\left\|2^{n_j}n_j!^2B_\omega^{n_j}\l \sum_{l=1}^\infty S_\omega^{n_l} \f{z_l}{2^{n_l}n_l!^2}+ \sum_{l=1}^\infty \sum_{k=n_{l}+l+1}^{n_{l+1}} \f{e_k}{k^2 2^{n_{l}}}\r-z_j\right\|\\
    &\leq \left \|\sum_{l={j+1}}^\infty S_\omega^{n_l-n_j} \f{z_l2^{n_j}n_j!^2}{2^{n_l}n_l!^2}\right\|+ \left\|
    \sum_{l=1}^\infty \sum_{k=n_l+l+1}^{n_{l+1}} \f{2^{n_j}n_j!^2B_\omega^{n_j}e_k}{k^2 2^{n_{l}}}\right\|\\
    &\leq \sum_{l={j+1}}^\infty \f{2^{n_{l-1}}n_{l-1}!^2}{2^{n_l}n_l!^2}\| S_\omega^{n_l-n_j} z_l\|+ \left\|
        \sum_{l=j}^\infty \sum_{k=n_l+l+1}^{n_{l+1}} \f{2^{n_j}n_j!^2}{k^2 2^{n_{l}}} \f{k!^2}{(n_j+k)!^2}e_{k-n_j}\right\|\\
     &\leq \sum_{l={j+1}}^\infty \f{2^{n_{l-1}}n_{l-1}!^2}{2^{n_l}n_l!^2}\| S_\omega^{n_l} z_l\|+ \sum_{l=n_j+j+1}^\infty \f{1}{l^2}\\
    &\leq \sum_{l={j+1}}^\infty \f{2^{n_{l-1}}n_{l-1}!^2  (n_l+l)^{2l}l^2}{2^{
    n_l}} + \sum_{l=n_j+j+1}^\infty \f{1}{l^2}\\
    &\leq 2 \sum_{l\geq j}\f{1}{l^2}\to 0.
\end{align*}
\end{proof}

For the proof of Theorem \ref{main Theorem} we will also need the following property of Fibonacci numbers.
\begin{lemma}\label{Lema a_n 1}
	Let $a_n$ be recurrently defined as $$a_n:=\begin{cases}
	a_1=1\\
	a_n=n-\sum_{j=1}^{n-1} a_{n-j} F_{2j+1},
	\end{cases}$$
	where $(F_j)_j$ is the usual Fibonacci sequence. Then $$a_n=1-\frac{n(n-1)}{2}=-\sum_{j=2}^{n-1}j.$$
\end{lemma}
\begin{proof}
	We work with the auxiliary sequence $b_n=\sum_{j=1}^{n-1} a_{n-j}F_{2j}$. We claim that for $n\geq 3$, $1-b_n=-\sum_{j=2}^{n-1}j$.
	When $n=3$, we have
	\begin{align*}
	b_3=\sum_{j=1}^{3-1} a_{3-j}F_{2j}&=a_1F_4+a_2F_2=1\cdot3+0\cdot1=3.
	\end{align*}
	Therefore $1-b_3=-2=-\sum_{j=2}^{3-1}j.$
	Suppose now that $n>3$, then
	\begin{align*}
	b_{n+1}&=\sum_{j=1}^{n} a_{n+1-j}F_{2j}=a_{n}+\sum_{j=2}^{n} a_{n+1-j}F_{2j-2}+\sum_{j=2}^{n}a_{n+1-j}F_{2j-1}\\
	&=a_n+b_{n-1}+(n-a_n)=n+b_{n-1}.
	\end{align*}
	Consequently, proceeding inductively, we get,
	\begin{align*}
	1-b_n&=1-n-b_{n-1}=-\sum_{j=2}^{n}j.
	\end{align*}
	Finally we obtain,
	\begin{align*}
	a_n&=n-\sum_{j=1}^{n-1} a_{n-j} F_{2j+1}=n-\sum_{j=1}^{n-1} a_{n-j} F_{2j}-\sum_{j=1}^{n-1} a_{n-j} F_{2j-1}\\
	&=n-b_n-(n-1-a_{n-1}+a_{n-1})=1-b_n.
	\end{align*}
\end{proof}

\begin{proof}[Proof of Theorem \ref{En Banach}.]
We compute the orbit generated by a pair of vectors $x,y$,
\begin{equation}\label{formula}
M^{ n}(x,y)=\begin{cases}
                  B_\omega^\frac{n+1}{2}(x) c_{n}(x,y)d_n(\omega)& \text{ if } n \text{ is odd;}\\
                  B_\omega^\frac{n}{2}(y) c_{n}(x,y)d_n(\omega)& \text{ if } n \text{ is even.}
                  \end{cases}
\end{equation}
Where $c_n(x,y)$ is defined as 
$$c_n(x,y)=\begin{cases}
            y_1^{F_n}x_2^{F_{n-1}}y_2^{F_{n-2}}x_3^{F_{n-3}}\ldots y_\f{n+1}{2}^{F_1} & \text{ if } n \text{ is odd};\\
             y_1^{F_n}x_2^{F_{n-1}}y_2^{F_{n-2}}x_3^{F_{n-3}}\ldots x_{\f{n}{2}+1}^{F_1} & \text{ if } n \text{ is even}.\\
           \end{cases} $$
           and 
$$d_n(\omega)= \begin{cases}
\omega_1^{F_{n+1}-1}\cdot\omega_2^{F_{n-1}-1}\cdot\omega_3^{F_{n-3}-1}\ldots 
\omega_\frac{n-1}{2}^{F_{4}-1}& \text{ if } n \text{ is odd};\\
\omega_1^{F_{n+1}-1}\cdot\omega_2^{F_{n-1}-1}\cdot\omega_3^{F_{n-3}-1}\ldots 
\omega_\frac{n}{2}^{F_{3}-1}& \text{ if } n \text{ is even}.
\end{cases}$$
The sequence $c_n(x,y)$ may be seen as $c_n(z)$, where $z$ is the merge between $x$ and $y$, $z=(y_1,x_2,y_2,x_3,\ldots)$ and $c_n(z)=z_1^{F_n}\ldots z_n^{F_1}$. As we saw in Section \ref{SecRes}, we have that $c_n(z)=c_{n-1}(z)c_{n-2}(z)z_n$.  
Similarly the sequence $d_n(\omega)$ satisfies the following recursive relation, for $n\ge 3$,
$$
d_n(\omega)=\begin{cases}
d_{n-1}(\omega)d_{n-2}(\omega)\omega_1\ldots\omega_\frac{n-1}{2}&  \text{ if } n \text{ is odd}\\
d_{n-1}(\omega)d_{n-2}(\omega)\omega_1\ldots\omega_\frac{n}{2} & \text{ if } n \text{ is even}.
\end{cases}
$$
Indeed, suppose now  $n\ge 3$  is even, then $n-1$ and $n+1$ are odd. By using the above formulas we get
\begin{align*}
d_n(\omega)d_{n-1}(\omega)\omega_1\ldots\omega_\frac{n+1-1}{2}&=\omega_1^{F_{n+1}-1}\omega_2^{F_{n-1}-1}\ldots\omega_\frac{n}{2}\cdot \omega_1^{F_{n-1+1}-1}\omega_2^{F_{n-1-1}-1}\ldots\omega_\frac{n-2}{2}^2\cdot\omega_1\ldots\omega_\frac{n}{2}\\
&=\omega_1^{F_{n+1}-1+F_{n}-1+1}\omega_2^{F_{n-1}-1+F_{n-2}-1}\ldots \omega_{\frac{n-2}{2}}^{F_5-1+F_4-1+1}
\omega_\frac{n}{2}^2\\
&=\omega_1^{F_{n+2}-1}\omega_2^{F_{n}-1}\ldots \omega_{\frac{n-2}{2}}^{F_6-1}
\omega_\frac{n}{2}^{F_4-1}.
\end{align*}
If $n$ is odd, it is analogous.

 Now we prove equality \eqref{formula} again by induction. 
 If $n=1$ we have that $c_1(x,y)=y_1^{F_1}=y_1$ and $d_1(\omega)=\omega_1^{F_2-1}=1$, while $M(x,y)=B_\omega (x)y_1$.
Suppose now that $n>1$ is even, then
\begin{align*}
M^{n}(x,y)&=M\l M^{ n-2}(x,y),M^{ n-1}(x,y)\r\\
&=M\l B_\omega^\frac{n-2}{2}(y)c_{n-2}(x,y)d_{n-2}(\omega),
B_\omega^\frac{n}{2}(x)c_{n-1}(x,y)d_{n-1}(\omega)\r\\ 
&=B_\omega\l B_\omega^\frac{n-2}{2}(y)c_{n-2}(x,y)d_{n-2}(\omega)\r\cdot \left[B_\omega^\frac{n}{2}(x)c_{n-1}(x,y)d_{n-1}(\omega)\right]_1\\
&=B_\omega^\frac{n}{2}(y)\cdot c_{n-2}(x,y)\cdot d_{n-2}(\omega)\cdot c_{n-1}(x,y)\cdot d_{n-1}(\omega)\cdot x_{\frac{n}{2}+1}\cdot\omega_1\ldots \cdot \omega_{\frac{n}{2}}\\
&=B_\omega^\frac{n}{2}(y) c_n(x,y)d_n(\omega).
\end{align*}
Again, the case when $n$ is odd is analogous.

Next, we will construct a pair of vectors $(x,y)$ such that the even iterations $(M^{2n}(x,y))_n$ are dense. The  sequence   $(c_{2n}(x,y)d_{2n}(\omega)B_\omega^n(y))_n$ can be seen as the product of a universal family with certain weights. So if we manage to control the weights so that  $y$ is an universal vector for this family, the orbit will be dense.

Let $y\in\ell_1$ be fixed, we want to find $x$ (depending on $y$) such that $c_{2n}(x,y)d_{2n}(\omega)B_\omega^n(\cdot)$ results universal. Observe that $B_\omega^n$ has norm $\frac{1}{n!^2}$, this implies that, in order to get an universal family, the searched weights must be of order higher than $n!^2$. Also $c_n$ is multiplicative, where the multiplication in $\mathbb C^\zN$ is coordinate-wise, that is $c_n(z\cdot w)=c_n(z)\cdot c_n(w)$. 


If we want to choose $x$ to cancel the weights induced by $\omega$ and $y$, one suitable vector $\tilde x$ is  
$$[\tilde x]_{i+1}=\omega_i^{-1} \prod_{l\leq i}y_l^{-1}\omega_l^{-1}.$$
In this particular case we get $c_{2n}(x,y)d_{2n}(\omega)=n!^2$. However, the family $\{n!^2B_\omega^n\}$ fails to be universal and there is no hope that $\tilde x$ is well defined. Therefore, we multiply pointwise $\tilde x$ by another sequence, $2^{a_n}$, to get $c_{2n}(\tilde x\cdot 2^{a_n},y)=2^{n}n!^2$. It turns out that $a_n$ is a polynomial of degree two with principal coefficient negative. This will allow us to construct $y$ such that $y$ is universal for $2^{n}n!^2B_\omega^n$ and that  $x$ is well defined. 
 
Since $c_n$ is multiplicative, $a_n$ must satisfy $c_n(2^{a_n},1)=2^{n}$. The sequence $(a_n)_n$ we need is the one defined in Lemma \ref{Lema a_n 1},  
$$a_n:=\begin{cases}
   a_1=1\\
   a_n=n-\sum_{j=1}^{n-1} a_{n-j} F_{2j+1}.
  \end{cases}.$$
We claim that, for a fixed vector $y$ the vector $x$ defined as 
\begin{equation}\label{x}
[x]_{i+1}=
2^{a_i}\omega_{i}^{-1}\prod_{j\leq i} y_j^{-1}\omega_{j}^{-1}
\end{equation} 
satisfies that 
$c_{2n}(x,y)d_{2n}(\omega)=2^{n}n!^2$. 
To show this equality we will use the following well known identity. 
\begin{equation}\label{eqfib2}
 F_{2n}=\sum_{j=1}^{n} F_{2j-1}.
\end{equation}
Next we prove our claim,  
\begin{align*}
c_{2n}(x,y)d_{2n}(\omega)&=y_1^{F_{2n}}\cdot x_2^{F_{2n-1}} \cdot y_2^{F_{2n-2}}\cdot\ldots \cdot x_{n}^{F_3}\cdot y_n\cdot x_{n+1}\cdot \omega_1^{F_{2n+1}-1}\cdot\omega_2^{F_{2n-1}-1}\cdot\omega_3^{F_{2n-3}-1}\ldots 
\omega_n\\
&=(\omega_1y_1)^{F_{2n}}\cdot(\omega_1 x_2)^{F_{2n-1}} \cdot (\omega_2y_2)^{F_{2n-2}}\cdot (\omega_2x_3)^{F_{2n-3}} \ldots \cdot (\omega_{n-1}x_{n})^{F_3}\cdot (\omega_n y_n)\cdot(\omega_n x_{n+1})\cdot \omega_1^{-1}\ldots \omega_{n}^{-1}\\ 
&=(\omega_1y_1)^{F_{2n}}\cdot(\omega_1 y_1^{-1}\omega_1^{-2})^{F_{2n-1}}2^{a_1F_{2n-1}}\cdot (\omega_2y_2)^{F_{2n-2}}\cdot \\
&\quad\cdot(\omega_2y_1^{-1}y_2^{-1}\omega_1^{-1}\omega_2^{-2})^{F_{2n-3}} 2^{a_2F_{2n-3}}\cdot\ldots \cdot (\omega_n y_n) \cdot (\omega_ny_1^{-1}\ldots y_n^{-1}\omega_{1}^{-1}\ldots \omega_{n-1}^{-1}\omega_{n}^{-2}\cdot 2^{a_n})\omega_1^{-1}\ldots\omega_{n}^{-1}\\
&=(\omega_1y_1)^{F_{2n}-\sum_{j=1}^{n}	 F_{2j-1}}(\omega_2y_2)^{F_{2(n-1)}-\sum_{j=1}^{n-1} F_{2j-1}}\ldots (\omega_{n-1}y_{n-1})^{F_4-F_3-F_1}\cdot y_n^{1-1}\\
&\quad\cdot 2^{a_n+\sum_{j=1}^{n-1} a_{n-j} F_{2j+1}} \omega_1^{-1}\ldots\omega_n^{-1}\\
&=(\omega_1y_1)^0\ldots (\omega_{n-1}y_{n-1})^0  2^{n}\prod_{l=1}^n \omega_l^{-1}=2^{n}n!^2.
\end{align*}

Therefore it suffices to find an universal vector $y$ for 
$2^{n}n!^2B_\omega^n$ so that its induced vector $x$ defined as in \eqref{x} is well defined. By Lemma \ref{Lema a_n 1}, 
$$
a_n=1-\frac{n(n-1)}{2},
$$ 
and by Theorem \ref{construccion 2} there exist vectors $x,y$ with the required properties.
\end{proof}

\subsubsection{Second step}
Recall that an $m$-linear operator $L\in \LL(^m X)$ is said to be quasiconjugated to  
an $m$-linear operator $N\in \LL(^mY)$ if there exists a continuous function $\phi:Y\to X$, with dense range, such that the following diagram commutes,
$$\xymatrix{
 Y^m \ar [r] ^N \ar[d]^{\phi^m}& Y \ar^\phi[d]\\
X^m \ar[r]^L& X&
},$$
where $\phi^m=\phi\times\ldots\times\phi$.
Analogously to \cite[Theorem 3]{GroKim13}, we have the following proposition.
\begin{proposition}\label{Conj}
Let $N$ be a hypercyclic multilinear operator. If an $m$-linear operator $L$ is quasiconjugated to $N$, then $N$ is also hypercyclic.
\end{proposition}
\begin{proof}
We will prove that $Orb_L(\phi(x_1),\ldots,\phi(x_m))=\phi(Orb_N(x_1,\ldots,x_m)$ for each $m$-tuple of vectors $x_1,\ldots,x_m$. It suffices to show that for each $j$, $L^{ j} (\phi(x_1),\ldots,\phi(x_m))=\phi \l N^{ j}(x_1,\ldots,x_m)\r$. We see this equality by induction. For $j=1$ it is clear since $L(\phi(x_1),\ldots,\phi(x_m))=\phi (N(x_1,\ldots,x_m))$, because $L$ is quasiconjugated to $N$. Suppose that our claim is true for each $i<j$, and suppose first that $j>m$, then
\begin{align*}
    L^{ j}(\phi(x_1),\ldots,\phi(x_m))&=L \l L^{ (j-m)} (\phi(x_1),\ldots,\phi(x_m)),\ldots, L^{ (j-1)}(\phi(x_1),\ldots, \phi(x_m))\r\\
    &= L\l \phi(N^{ (j-m)}(x_1,\ldots,x_m),\ldots,\phi(N^{ (j-1)}(x_1,\ldots,x_m)\r\\
    &=\phi \l N\l N^{ (j-m)}(x_1,\ldots,x_m),\ldots,N^{ (j-1)}(x_1,\ldots,x_m)\r\r\\
    &= \phi \l N^{ j}(x_1,\ldots,x_m)\r.
\end{align*}
If $j\leq m$ the proof is analogous.
Finally, let $(x_1,\ldots x_m)$ be an hypercyclic tuple. Since $Orb_N(x_1,\ldots,x_m)$ is dense and $\phi$ has dense range, it follows that $(\phi(x_1),\ldots,\phi(x_m))$ is hypercyclic for $L$.
\end{proof}
\subsubsection{Final step}
To prove the existence Theorem \ref{main Theorem}, we will use the following result proved independently by Pelczy\'nski and Plichko (see, \cite[Theorem 1.27]{HaPeMon08}).

\begin{lemma}\label{Pelc}
	Let $X$ be an infinite dimensional separable Banach space. Then, for any $\varepsilon>0$ there exists  $(1+\varepsilon)$-bounded Markushevich basis. That is, there exist a sequence $(x_n,x^*_n)_n\subset X\times X^*$ such that
	\begin{enumerate}
		\item $span\{x_n:n\in\mathbb N\}$ is dense in $X$, $span\{x_n^*:n\in\mathbb N\}$ is $w^*$-dense in $X^*$
		\item $\sup_n\|x_n\|\cdot\|x_n^*\|< 1+\varepsilon$.
		\item $x_n^*(x_k)=\delta_{n,k}$.
	\end{enumerate} 
\end{lemma}

\begin{theorem}\label{quasiconj en cualquier Frechet}
 Let $X$ be an infinite dimensional separable Banach  space. Let $\omega$ be $(1,\f{1}{2^2},\f{1}{3^2},\f{1}{4^2},\ldots)$ and let $M$ be the bilinear operator $e_1'(y)B_{\omega}(x)\in \mathcal L (^2\ell_1;\ell_1)$, that is $[M(x,y)]_i=y_1\cdot \omega_i x_{i+1}$. Then there exists a bilinear operator $N\in\mathcal L (^2 X;X)$ such that $N$ is quasiconjugated to $M$.
\end{theorem}
\begin{proof}
Let $(x_n)_n$ and $(x_n^*)_n$ be sequences given by the above lemma. Without loss of generality we may suppose that $\|x_n\|=1$ and $\|x_n^*\|\leq 2$. Let $N$ be the bilinear operator
$$N(u,v)=x_1^*(v)\sum_{l=1}^\infty x_l^*(u)\frac{1}{(l-1)^2} x_{l-1}.$$
Since $\|x_n^*\|\leq 2$ and $\|x_n\|=1$ it follows that $N$ is a well defined continuous bilinear operator on $X$. We consider now the factor $\phi:\ell_1\to X$, $\phi((a_n)_n)=\sum_l a_lx_l$. Again, $\phi$ is a well defined continuous linear operator because $\|x_n\|\leq 1$ and $(a_n)_n\in\ell_1$. Observe that $\phi$ has dense range, because $x_n=\phi(e_n)$ and thus $  X=\overline {span(\{x_n\})}\subseteq\overline{R(\phi)}.$ It remains to see that $N$ is quasiconjugated to $M$ via $\phi$. Since $\phi$ is linear, it suffices to check the commutative relation for elements in the canonical basis of $\ell_1$. If $e_k$ and $e_j$ are elements of the basis we have that 
\begin{align*}
\phi(M(e_k,e_j))&=\phi(e_1^*(e_j)\frac{1}{(k-1)^2}e_{k-1})\\
&=\delta_{1,j}\frac{1}{(k-1)^2}x_{k-1}.
\end{align*}
On the other hand we have that,
\begin{align*}
N(\phi(e_k),\phi(e_j))&=N(x_k,x_j)=x_1^*(x_j)\cdot\sum_{l=1}^\infty x_l^*(x_k)\frac{1}{(l-1)^2}x_{l-1}\\  
&=\delta_{1,j} \frac{1}{(k-1)^2} x_{k-1}.
\end{align*}
\end{proof}

\begin{proof}[Proof of Theorem \ref{main Theorem}]
 Apply Theorems \ref{En Banach}, \ref{quasiconj en cualquier Frechet} and Proposition \ref{Conj}.
\end{proof}

\section{Existence of symmetric bihypercyclic operators on arbitrary Banach spaces}\label{SecBih}

 Recall that the orbit induced by a bilinear operator in the sense of Grosse-Erdmann and Kim \cite{GroKim13} with initial conditions $x,y$ is  $\cup_{n\ge 0} M^n(x,y)$ where $M^0(x,y)=\{x,y\}$ and the $n$-states are inductively defined as $M^n(x,y)=M^{n-1}(x,y)\cup \{M(z,w):z,w\in M^{n-1}(x,y)\}$. A bilinear operator is said to be \textit{bihypercyclic} if the orbit with initial conditions $x,y$,  $\cup_{n\in\zN_0} M^n(x,y)$ is dense in $X$. 

As in the cases of homogeneous polynomials and multilinear operators in the sense of B\`es and Conejero there is a notion of limit ball: if $x,y\in \f{1}{\|M\|}B_X\times \f{1}{\|M\|}B_X$, then $M^n(x,y)\sub \f{1}{\|M\|}B_X$. Moreover, in this case, the orbit tends to zero, i.e. for every open set $U$ there is some $n_0$ such that $M^n(x,y)\sub U$ for every $n\geq n_0$. Thus, the set of bihypercyclic vectors in a Banach space is never residual. Despite this restrictive fact, in \cite{GroKim13}, it was observed that if $T$ is a hypercyclic operator and $x^*$ is a nonzero linear funtional the bilinear mapping $x^*\otimes T$ is bihypercyclic, and thus there are bihypercyclic bilinear operators in arbitrary infinite dimensional separable Banach spaces. They also proved that there are bihypercyclic bilinear mappings in the finite dimensional case.  However it is unknown whether the operator can be taken to be symmetric and the following question was posed (see \cite[p. 708]{GroKim13}).

\textbf{Question A.} \textit{
Let $X$ be a separable Banach space. Does there exist a symmetric bihypercyclic operator in $\LL(^2X)$?}


We will prove that there are bihypercyclic symmetric operators on arbitrary infinite dimensional separable Banach spaces. The main tool to produce bihypercyclic bilinear operators in \cite{GroKim13} is to construct a bilinear operator $M$ such that $T(\cdot)=M(\cdot,y)$ is a hypercyclic linear operator for some $y\in X$, because in this case the orbit of $x$ by $T$ is contained in the orbit of $(x,y)$ by $M$, and thus  it follows that $M$ is bihypercyclic. We will follow here a different approach and study the orbit of the homogeneous polynomial induced by $M$. It is known that homogeneous polynomials on Banach can not be hypercyclic \cite{Ber98}, so that the subset of the orbit of $M$ with initial conditions $(x,x)$ given by $\{P^n(x):n\in\zN\}$ is never dense. However, we can still achieve $\overline{\{M(P^n(x),P^m(x)):n,m\in \zN\} }=X$. The structure of the proof will be the same as the one used to prove Theorem \ref{main Theorem}. We will look for a symmetric bihypercyclic bilinear operator $M$ such that it quasiconjugates to arbitrary separable and infinite dimensional Banach spaces. Our candidate will be the symmetrization of the bilinear operator considered in Theorem \ref{En Banach}, $M(x,y)= \f{e_1'(x)B_\omega(y)+e_1'(y)B_\omega(x)}{2}$. We will show that if $P$ is the homogeneous polynomial induced by $M$, $P(x)=M(x,x)=e_1'(x)B_\omega(x)$, then there is some vector $x$ for which $\overline{\{M(P^n(x),P^m(x)):n,m\in \zN\}}=\ell_1$. This polynomial $P$ was studied in  \cite{CardMurPoli}. We will need the following definitions and results, which were posed in \cite{CardMurPoli}. Given a homogeneous polynomial $Q$ acting on a Banach space it is worth considering its Julia set $J_Q=\partial \{x:\lim Q^n(x)=0\}$. This sets are always closed, perfect and invariant. Moreover, if the set $\{x:\lim Q^n(x)=0\}$ is dense in the space then $J_Q$ is completely invariant. In the particular case of our polynomial $P=e_1'B_w$ (see \cite[Section 4]{CardMurPoli}), we have that $P|_{J_P}:J_P\to J_P$ is transitive so that by the Birkhoff's Transitivity Theorem there is a vector $x$ such that $\overline{\{P^n(x):n\in\mathbb N\}}=J_P$. This Julia set satisfies also the following properties:

\begin{lemma}
For every $j\in \zN$, there is $n_0(j)$ such that for every $n\geq n_0(j)$, the sequence $\l n,\f{1}{j!^2},\f{1}{(j+1)!^2},\ldots \r$ is in $J_P$.
\end{lemma}
\begin{lemma}
If $y\in J_P$ and for every coordinate $i$, $|x_i|\ge|y_i|$, then $x\in J_P$.
\end{lemma}
 Applying this results it follows easily that $M$ is bihypercyclic.
\begin{theorem}\label{symmetric}
Let $M(x,y)= \f{e_1'(x)B_\omega(y)+e_1'(y)B_\omega(x)}{2}$, where $[B_\omega(x)]_i= \f{x_{i+1}}{i^2}$. Then $M$ is a bihypercyclic bilinear mapping on $\ell_1$.
\end{theorem}
\begin{proof}
By the comments above, there is a vector $x$ such that $\overline{\{P^n(x):n\in\mathbb N\}}=J_P$, so it suffices to show that $\overline{M(J_P,J_P)}=\ell_1$. Let $x_0\in c_{00}$. By the above lemmas  and by the completely invariance of $J_P$ there exists $\lambda>0$ such that $x=\lambda S_\omega(x_0) + \sum_{i=2}^\infty \f{1}{i^2}e_i+e_1\in J_P$ and $y=-(\lambda-2) S_\omega(x_0) - \sum_{i=2}^\infty \f{1}{i^2}e_i+e_1 \in J_P$, where $S_\omega$ is the formal right inverse of $B_\omega$. Now,
 \begin{align*}
     M(x,y)&=\f{x_1B_\omega(y)+y_1B_\omega(x)}{2}=\f{B_\omega(y)+B_\omega(x)}{2}\\
     &=\f{(\lambda-\lambda+2) x_0}{2}=x_0.
 \end{align*}
\end{proof}

With the same tools that we used in Theorem \ref{main Theorem} we can prove the following. 
\begin{theorem}\label{bihiperexist}
Let $X$ be an infinite dimensional separable Banach space. Then there exists a symmetric bihypercyclic bilinear mapping $A\in \LL_s(^2X)$.
\end{theorem}
\begin{proof}
 Let $\phi$ and $N$ be the operators defined in the proof of Theorem \ref{quasiconj en cualquier Frechet} and $M$ the bihypercyclic symmetric bilinear defined in the proof of Theorem \ref{symmetric}. Then, if $A$ is the symmetrization of $N$ it follows that $N$ is a quasiconjugation of $M$ via the factor $\phi$. Therefore, since $M$ is bihypercyclic we obtain that $A$ is also bihypercyclic.
\end{proof}
\section{Final comments}
We would like to end the paper with two questions on non normable Fr\'echet spaces. Our proofs of Theorems \ref{main Theorem} and \ref{bihiperexist} are based on the existence of a bounded biorthogonal basis. However for non normable Fr\'echet spaces this result does not hold.

\textbf{ Problem 1.} Does every infinite dimensional and separable Fr\'echet space admit an $m$-multilinear hypercyclic operator? And a strongly transitive multilinear operator?

\textbf{ Problem 2.} Does every infinite dimensional and separable Fr\'echet space admit a symmetric bilinear bihypercyclic operator?
\section*{acknowledgement}
The author wishes to express his gratitude to Santiago Muro for his support and guidance as supervisor. His fruitful comments improved the presentation of the paper.


\begin{thebibliography}{10}
\bibliographystyle{abbrv}
\bibitem{ansari1997existence}
S.~I. Ansari.
\newblock Existence of hypercyclic operators on topological vector spaces.
\newblock {\em {J}ournal of {F}unctional {A}nalysis}, 148(2):384--390, 1997.

\bibitem{AroMir08}
R.~M. {Aron} and A.~{Miralles}.
\newblock {Chaotic polynomials in spaces of continuous and differentiable
  functions.}
\newblock {\em {Glasg. Math. J.}}, 50(2):319--323, 2008.

\bibitem{BayMat09}
F.~Bayart and E.~Matheron.
\newblock {\em {Dynamics of linear operators.}}
\newblock {Cambridge Tracts in Mathematics 179. Cambridge: Cambridge University
  Press. xiv, 337~p.}, 2009.

\bibitem{bernal1999hypercyclic}
L.~Bernal-Gonz{\'a}lez.
\newblock On hypercyclic operators on {B}anach spaces.
\newblock {\em Proceedings of the American Mathematical Society},
  127(4):1003--1010, 1999.

\bibitem{Ber98}
N.~C. {Bernardes}.
\newblock {On orbits of polynomial maps in Banach spaces.}
\newblock {\em {Quaest. Math.}}, 21(3-4):311--318, 1998.

\bibitem{BesCon14}
J.~B{\`e}s and J.~A. Conejero.
\newblock An extension of hypercyclicity for n-linear operators.
\newblock In {\em Abstract and Applied Analysis}, volume 2014. Hindawi
  Publishing Corporation, 2014.

\bibitem{Bir29}
G.~D. Birkhoff.
\newblock {D\'emonstration d'un th\'eor\`eme \'el\'ementaire sur les fonctions
  enti\`eres.}
\newblock {\em C. R.}, 189:473--475, 1929.

\bibitem{BonePer98}
J.~{Bonet} and A.~{Peris}.
\newblock {Hypercyclic operators on non-normable Fr\'echet spaces.}
\newblock {\em {J. Funct. Anal.}}, 159(2):587--595, 1998.

\bibitem{CardMurPoli}
R.~Cardeccia and S.~Muro.
\newblock {Orbits of homogeneous polynomials on Banach spaces}.
\newblock {\em Preprint, arXiv:1806.11543.}

\bibitem{CardMurH(C)}
R.~Cardeccia and S.~Muro.
\newblock Hypercyclic homogeneous polynomials on {$H (C)$}.
\newblock {\em Journal of Approximation Theory}, 226:60--72, 2018.

\bibitem{Gro99}
K.-G. Grosse-Erdmann.
\newblock Universal families and hypercyclic operators.
\newblock {\em Bull. Amer. Math. Soc. (N.S.)}, 36(3):345--381, 1999.

\bibitem{GroKim13}
K.-G. Grosse-Erdmann and S.~G. Kim.
\newblock Bihypercyclic bilinear mappings.
\newblock {\em Journal of Mathematical Analysis and Applications},
  399(2):701--708, 2013.

\bibitem{GroPer11}
K.-G. Grosse-Erdmann and A.~Peris~Manguillot.
\newblock {\em {Linear chaos.}}
\newblock {Universitext. Berlin: Springer. xii, 386~p. EUR~53.45 }, 2011.

\bibitem{HaPeMon08}
P.~H\'{a}jek, V.~Montesinos~Santaluc\'{\i}a, J.~Vanderwerff, and V.~Zizler.
\newblock {\em Biorthogonal systems in {B}anach spaces}.
\newblock Springer, New York, 2008.

\bibitem{Mac52}
G.~R. MacLane.
\newblock Sequences of derivatives and normal families.
\newblock {\em J. Analyse Math.}, 2:72--87, 1952.

\bibitem{MarPer10}
F.~Mart{\'\i}nez-Gim{\'e}nez and A.~Peris.
\newblock Chaotic polynomials on sequence and function spaces.
\newblock {\em International Journal of Bifurcation and Chaos},
  20(09):2861--2867, 2010.

\bibitem{MeiVog97}
R.~Meise and D.~Vogt.
\newblock {\em Introduction to functional analysis}.
\newblock Clarendon Press, 1997.

\bibitem{Per99}
A.~{Peris}.
\newblock {Chaotic polynomials on Fr\'echet spaces.}
\newblock {\em {Proc. Am. Math. Soc.}}, 127(12):3601--3603, 1999.

\bibitem{Per01}
A.~{Peris}.
\newblock {Erratum to: ``Chaotic polynomials on Fr\'echet spaces''.}
\newblock {\em {Proc. Am. Math. Soc.}}, 129(12):3759--3760, 2001.

\bibitem{Rol69}
S.~{Rolewicz}.
\newblock {On orbits of elements.}
\newblock {\em {Stud. Math.}}, 32:17--22, 1969.

\bibitem{Vaj08}
S.~Vajda.
\newblock {\em Fibonacci and Lucas numbers, and the golden section: theory and
  applications}.
\newblock Courier Corporation, 2008.

\end{thebibliography}
\end{document}